\documentclass[14pt]{amsart}
\usepackage{cases}
\usepackage{amsmath}
\usepackage{amsfonts}
\usepackage{bm}
\usepackage{amsfonts,amsmath,amssymb,amscd,bbm,amsthm,mathrsfs,dsfont}
\usepackage{mathrsfs}
\usepackage{pb-diagram}
\usepackage{color}
\usepackage{amssymb}
\usepackage{xypic}
\usepackage[all]{xy}
\usepackage{mathrsfs}
\usepackage{amsthm}
\usepackage[all]{xy}
\usepackage{epsfig}
\usepackage{graphics}
\usepackage{array}
\usepackage{graphicx} 
\usepackage{epstopdf}
\usepackage{float}
\usepackage{tikz}
\usepackage{tikz-cd}
\usetikzlibrary{arrows}
\usetikzlibrary{graphs}
\usepackage{tikz}
\usepackage{hyperref}

\newtheorem{Theorem}{Theorem}[section]
\newtheorem{Lemma}[Theorem]{Lemma}
\newtheorem{Definition}[Theorem]{Definition}
\newtheorem{Corollary}[Theorem]{Corollary}
\newtheorem{Proposition}[Theorem]{Proposition}
\newtheorem{Example}[Theorem]{Example}
\newtheorem{Remark}[Theorem]{Remark}

\title {On maximal green sequences in abelian length categories}

\author{Fang Li $\;\;\;\;\;\;$ Siyang Liu $\;\;\;\;\;\;$}
\address{Fang Li
\newline Department of Mathematics, Zhejiang University (Yuquan Campus), Hangzhou, Zhejiang
310027,  P.R.China}
\email{fangli@zju.edu.cn}

\address{Siyang Liu
\newline Department
of Mathematics, Zhejiang University (Yuquan Campus), Hangzhou, Zhejiang
310027, P.R.China}
\email{siyangliu@zju.edu.cn}

\date{version of \today}

\newcommand{\lra}{\longrightarrow}

\newcommand{\ra}{\rightarrow}
\newcommand{\sdp}{\times\kern-.2em\vrule height1.1ex depth-.05ex}
\newcommand{\epi}{\lra \kern-.8em\ra}

\makeatletter
\@addtoreset{equation}{section}
\makeatother

 \normalbaselines
\begin{document}
\renewcommand{\thefootnote}{\alph{footnote}}
\setcounter{footnote}{-1} 

\renewcommand{\thefootnote}{\alph{footnote}}
\maketitle
\bigskip
\begin{abstract}
In this article, we study the relationship among  maximal
 green sequences, complete forward hom-orthogonal sequences and stability functions in abelian length categories.  Mainly, we firstly give a one-to-one correspondence between maximal  green sequences and  complete forward hom-orthogonal sequences via mutual constructions, and then prove that a maximal green sequence can be induced by a
 central charge if and only if it satisfies crossing inequalities.  As applications, we show that  crossing inequalities can be computed by $c$-vectors for finite dimensional algebras;  finally, we give the Rotation Lemma for finite dimensional Jacobian algebras.
\end{abstract}

\section{Introduction and preliminaries }
The maximal green sequence was originally defined to be a particular sequence of mutations of framed cluster quivers, which
 was firstly introduced by Keller in \cite{Kel}. Maximal green sequences are not only an important subject in cluster
 algebras, but also have important applications in many other objects, such as counting BPS states
 in string theory, producing quantum dilogarithm identities and computing refined Donaldson-Thomas
 invariants.

Cluster algebras have closed relations with representation theory via categorification, it follows that maximal
 green sequence could be interpreted from the viewpoint of tilting theory and silting theory. For example,
 a maximal green sequence for a cluster quiver corresponds to a sequence of forward mutation of a specific
 heart to its shift in a particular triangulated category. We refer to \cite{BDP} for more details. Inspired
 by $\tau$-tilting theory, Br$\mathrm{\ddot{u}}$stle, Smith and Treffinger defined maximal green sequence as
 particular finite chain of torsion classes for a finite dimensional algebra in \cite{BST0}, which can be
 also naturally defined in arbitrary abelian categories in \cite{BST1}.

Throughout this paper we always assume $\mathcal{A}$ is  a small abelian category.

Let us firstly recall some basic concepts.  Suppose $X$ is an object in $\mathcal{A}$.
 We say that $X$ has finite length, if there exists a finite filtration
 \[0=X_0 \subset  X_1 \subset X_2 \subset \dots \subset X_m =X\]
 such that $X_i/X_{i-1}$ is simple for all $i$. Such a filtration is called a {\bf{\em Jordan-H$\ddot{o}$lder}
 series} of $X$. It is well-known that if $X$ has finite length, then the length of the
 {\em Jordan-H$\ddot{o}$lder} series of $X$ is uniquely determined by $X$, which will be denoted
 by $l(X)$. Recall that an {\bf abelian length category} is an abelian category such that every object has
 finite length. Throughout this article, we always assume that $\mathcal{A}$ is an abelian length category.

Let $\mathcal{A}$
 be an abelian length category, and $\mathcal{T}$ and $\mathcal{F}$ be full subcategories of $\mathcal{A}$
 which are closed under isomorphisms. The pair $(\mathcal{T}, \mathcal{F})$ is called a {\bf torsion pair} if it satisfies
 the following conditions.
 \begin{enumerate}
     \item[(i)] For any objects $X\in \mathcal{T}$ and $Y\in \mathcal{F}$, then $Hom(X, Y) = 0$,
     \item[(ii)] An obeject $X$ belongs to $\mathcal{T}$  if and only if $Hom(X, Y)=0$ for any object $Y\in \mathcal{F}$,
     \item[(iii)] An obeject $Y$ belongs to $\mathcal{F}$  if and only if $Hom(X, Y)=0$ for any object $X\in \mathcal{T}$.
   \end{enumerate}

  For a torsion pair
 $(\mathcal{T}, \mathcal{F})$, the full subcategories $\mathcal{T}$ and $\mathcal{F}$ are called a  {\bf torsion class}
 and a {\bf torsion-free class}, respectively. It is well-known that a full subcategory in $\mathcal{A}$ is a torsion class
 if and only if it is closed under extensions and factors, and a full subcategory in $\mathcal{A}$ is a torsion-free
 class if and only if it is closed under extensions and subobjects. One of important properties of a torsion pair is that for any
 object $X$ in $\mathcal{A}$,  there is a unique exact sequence $0 \rightarrow X_1 \rightarrow X
 \rightarrow X_2 \rightarrow 0$ with $X_1\in \mathcal{T}$ and $X_2\in \mathcal{F}$ up to isomorphism,
 which is called the {\em canonical sequence} for $X$ with respect to the torsion pair $(\mathcal{T}, \mathcal{F})$.

Let $\mathcal{T}$ and $\mathcal{T}'$ be two torsion classes in $\mathcal{A}$. We say the torsion
 class $\mathcal{T}'$ {\em covers} $\mathcal{T}$ if $\mathcal{T} \subsetneq \mathcal{T}'$ and $\mathcal{X} = \mathcal{T}$ or $\mathcal{X} = \mathcal{T}'$ for any  torsion class $\mathcal{X}$ satisfying $\mathcal{T} \subset \mathcal{X} \subset \mathcal{T}'$.
 In this case, we write $\mathcal{T} \lessdot \mathcal{T}'$.

\begin{Definition}[\cite{BST1}] A {\bf maximal green sequence} in an abelian length category $\mathcal{A}$ is a finite
 sequence of torsion classes with covering relations
 \[\,\mathcal{T}_0\,\lessdot\, \mathcal{T}_1\, \lessdot\, \mathcal{T}_2\, \lessdot \,\dots\,
                             \lessdot\, \mathcal{T}_m\]
 such that $\mathcal{T}_0 = 0$ and $\mathcal{T}_m = \mathcal{A}$.
\end{Definition}

Stability conditions and Harder-Narasimhan filtration are widely studied by many authors
 and are very active. They were introduced in different contexts. For examples, King introduced stability
 functions on quiver representations in \cite{King}, and Rudakov extended it to abelian categories in \cite{Rud}.
 Let us recall basic definitions on stability functions and the important Harder-Narasimhan property for abelian
 length categories from \cite{Rud}.

\begin{Definition}[\cite{Rud, BST1}]
Let $\mathcal{P}$ be a totally ordered set and $\phi : \mathcal{A}^*\rightarrow \mathcal{P}$ a
 function on $\mathcal{A}^*=\mathcal{A}\backslash \{0\}$ which is constant on isomorphism classes.
 The map $\phi$ is called a {\bf stability function} if for each short exact sequence
                  $0\rightarrow L\rightarrow M\rightarrow N\rightarrow 0$
 of nonzero objects in $\mathcal{A}$ one has the so-called {\bf see-saw property}:

 either $\phi(L) = \phi(M) = \phi(N)$,

  or  $\phi(L) > \phi(M) > \phi(N)$,

   or $\phi(L) < \phi(M) < \phi(N)$.

Moreover, a nonzero object
 $M$ in  is said to be {\bf $\phi$-stable} (or {\bf $\phi$-semistable}) if every nontrivial subobject
 $L \subset M$ satisfies $\phi(L) < \phi(M)$ (or $\phi(L) \leq \phi(M)$, respectively).
\end{Definition}

Let $\phi$ be a stability function on $\mathcal{A}$. For any nonzero object $X$ in $\mathcal{A}$, we call $\phi(X)$ the {\bf phase} of $X$.
 When there is no confusion, we will simply call an object {\bf semistable} (respectively, {\bf stable})
 instead of $\phi$-semistable (respectively, $\phi$-stable). Rudakov proved the Harder-Narasimhan property as follows.

\begin{Theorem}[\cite{Rud}]\label{ruda}
Let $\phi : \mathcal{A}\rightarrow \mathcal{P}$ be a stability function, and let $X$ be a
 nonzero object in $\mathcal{A}$. Then up to isomorphism, $X$ admits a unique Harder-Narasimhan filtration,
  that is a filtration
 \[0=X_0\subsetneq X_1\subsetneq X_2 \subsetneq \dots \subsetneq X_l =X\]
 such that the quotients $F_i = X_i/X_{i-1}$ are semistable,
 and $\phi(F_1) > \phi(F_2) > \dots > \phi(F_l)$.

On the other hand, if $Y$ is a semistable object in $\mathcal{A}$, then there exists a filtration of $Y$
  \[0=Y_0\subsetneq Y_1\subsetneq Y_2 \subsetneq \dots \subsetneq Y_m =Y\]
such that the quotients $G_i = Y_i/Y_{i-1}$ are stable, and $\phi(Y) = \phi(G_m) = \dots =\phi(G_1)$.
\end{Theorem}

The second part of Theorem \ref{ruda} claims that any semistable object admits a stable subobject and a
 stable quotient with the same phase as the semistable object.
 Following from \cite{BST1}, we call $F_1 = X_1$ the {\bf maximally destabilizing subobject} of $X$ and
 $F_l = X_l/X_{l-1}$ the {\bf maximally destabilizing quotient} of $X$. They are unique up to isomorphism.

For a stability function $\phi : \mathcal{A}\rightarrow \mathcal{P}$, T. Br$\mathrm{\ddot{u}}$stle, D. Smith, and H. Treffinger proved in \cite{BST1} that  it can induce a torsion pair $(\mathcal{T}_p, \mathcal{F}_p)$ in $\mathcal{A}$ for every $p\in \mathcal{P}$ which
 is given as follows.
 \[ \mathcal{T}_{\geq p}\, =\, \{X\,\in\,Obj(\mathcal{A}): \phi(X')\,\geq\, p\,\, \text{for the maximally destabilizing
                          quotient $X'$ of $X$}\} \cup \{0\},\]
 \[ \mathcal{F}_{<p}\, =\, \{Y\,\in\,Obj(\mathcal{A}): \phi(Y'')\,<\, p\,\, \text{for the maximally destabilizing
                          sub-object $Y''$ of $Y$}\} \cup \{0\}.\]
 Then
  $\{\mathcal{T}_{\geq r}\}_{r\in \mathcal{P}}$ is called
 the {\bf chain of torsion classes induced the stability function}
 $\phi$. Furthermore, if $\{\mathcal{T}_{\geq r}\}_{r\in \mathcal{P}}$ forms a maximal green sequence, it called the
  {\bf maximal green sequence induced by the stability function}  $\phi$.

 Note that for any $r, s \in \mathcal{P}$, we have that $\mathcal{T}_{\geq r} \subset \mathcal{T}_{\geq s}$ if and only
 if $r \geq s$, and $\mathcal{T}_{\geq r} \subsetneq \mathcal{T}_{\geq s}$ implies $r > s$.
   In \cite{BST1}, Br$\mathrm{\ddot{u}}$stle, Smith and Treffinger proved that
 under some conditions on the stability function, the chain of torsion classes induced by the stability function is
 a maximal green sequence in $\mathcal{A}$.

 On the other hand, the important examples of stability functions are given by central charges. Let $\mathcal{A}$ be
 an abelian length categories with exactly $n$ nonisomorphic simple objects $S_1, S_2, \dots, S_n$. We know that the  Grothendieck
 group $K_0(\mathcal{A})$ of $\mathcal{A}$ is isomorphic to $\mathbb{Z}^n$.
 \begin{Definition}
A {\bf central charge} $Z$ on $\mathcal{A}$ is an additive map $Z: K_0(\mathcal{A}) \rightarrow \mathbb{C}$
which is given by \[Z(X) = \langle\alpha, [X]\rangle + \mathrm{i}\langle\beta, [X]\rangle\]
 for $X\in Obj(\mathcal{A})$. Here $\alpha\in \mathbb{R}^n$ and $\beta\in \mathbb{R}_{>0}^n$ are fixed, and $\langle\cdot \,, \cdot\rangle$ is the canonical inner product on $\mathbb{R}^n$ and $\mathrm{i}= \sqrt{-1}$.
\end{Definition}

Since $\langle\beta, [X]\rangle > 0$ for any nonzero object $X$ in $\mathcal{A}$, then $Z(X)$ lies in the strict upper half space of the complex space.
 It is well-known that every central charge $Z$ on $\mathcal{A}$ determines a stability function $\phi_{Z}$ (see also the proof of Theorem \ref{main}),  which is given by
\[\phi_Z(X) = \frac{argZ(X)}{\pi}.\]

We say that  a maximal green sequence can be induced by a central charge if the stability function determined by a central charge induces this maximal green sequence.\\

This article is organized as follows.
 In Section \ref{2}, we study relations between maximal green sequences and  complete forward hom-orthogonal sequences.
 In Section \ref{3.1}, we study properties of maximal green sequences induced
 by stability functions. In Section \ref{3.2}, we define {\bf crossing inequalities} for maximal green sequences (see Definition \ref{as}),
 and then prove the following main result.

{\bf Theorem \ref{main}}\;
{\em    A maximal green sequence  $\mathcal{T}: 0 = \mathcal{T}_0\, \lessdot\, \mathcal{T}_1\, \lessdot\, \mathcal{T}_2\, \lessdot \,\dots\, \lessdot\, \mathcal{T}_m = \mathcal{A} $ in an abelian length category $\mathcal{A}$ is induced by some central charge $Z : K_0(\mathcal{A}) \rightarrow \mathbb{C}$ if and only if
 $\mathcal{T}$ satisfies crossing inequalities. }

In Section \ref{4.1}, for finite dimensional algebras, we formulate relations between maximal green sequences of torsion classes and maximal green sequences of $\tau$-tilting pairs, which are defined via $c$-vectors. In Section \ref{4.2}, we prove the Rotation Lemma for finite dimensional Jacobian algebras and apply Theorem \ref{main} to formulate relations of crossing inequalities between Jacobian algebras and its mutation.

\section{Correspondence between maximal green sequences and  complete forward hom-orthogonal sequences}\label{2}

\subsection{Complete forward hom-orthogonal sequences}
We recall the concept of complete forward hom-orthogonal sequences from \cite{Ig1, Ig2}.  Let us introduce some notations. Let $\mathcal{A}$
 be an abelian length category, $\mathcal{C}$ be a subcategory of $\mathcal{A}$  and $N$ be an object in $\mathcal{A}$.
 A {\bf wide subcategory} of $\mathcal{A}$ is an abelian subcategory closed under extensions.
  The full subcategory $N^{\bot}$ is defined to be $N^{\bot} : = \{X\in \mathcal{A} | Hom(N,X)=0\}$
 and the full subcategory $\mathcal{C}^{\bot}$ is defined to be $\mathcal{C}^{\bot} : = \{X\in \mathcal{A} | Hom(Y,X)=0, \forall Y\in \mathcal{C}\}$.
 The full subcategories $^{\bot}N$ and $^{\bot}\mathcal{C}$ are defined similarly.
  We also write $\mathcal{F}(N):= N^{\bot}$ and $\mathcal{G}(N):=
 {}^\bot{\mathcal{F}(N)}$ for every object $N\in Obj(\mathcal{A})$.

 Then it is clear that $\mathcal{F}(N) = \mathcal{G}(N)^{\bot}$ and
 $(\mathcal{F}(N),\,  \mathcal{G}(N))$ is a torsion pair in $\mathcal{A}$.

\begin{Proposition}[\cite{Ig2}]\label{gx}
 Suppose that $Hom(X, Y) = 0$ and $\mathcal{C} = X^{\bot} \cap {}^{\bot}Y$, then $\mathcal{G}(X) =
                         {}^{\bot}\mathcal{C} \cap {}^{\bot}Y$.
\end{Proposition}

\begin{Definition}
An object $X$ in $\mathcal{A}$ is called a {\bf brick}, if $EndX$ is a division ring.
\end{Definition}

It is obvious that any brick is indecomposable. Let $\mathcal{S}$ be a subset of $obj(\mathcal{A})$, we use $Filt(\mathcal{S})$ to denote
the full subcategory of $\mathcal{A}$ consisting of objects having a finite filtration with subquotients
are isomorphic to indecomposable objects in $\mathcal{S}$, i.e., $X\in Filt{\mathcal{S}}$ if and only if there exists a finite filtration of
$X$: \[0=X_0 \subset  X_1 \subset X_2 \subset \dots \subset X_m =X\]
such that $X_i/X_{i-1}\in Ind(\mathcal{S})$ for all $i$. For an indecomposable object $X$, we will denote $Filt(\{X\})$ by  $Filt(X)$.

 The following lemma is well-known.

\begin{Lemma}[\cite{Rin}]
If $X$ is a brick in $\mathcal{A}$, then $Filt(X)$ is a wide subcategory of $\mathcal{A}$.
\end{Lemma}

\begin{Definition}[\cite{Ig1, Ig2}]\label{cfho}
A {\bf complete forward hom-orthogonal sequence} (briefly, CFHO sequence) in $\mathcal{A}$ is a finite sequence of bricks $N_1, N_2, \dots , N_m$
 such that
 \begin{enumerate}
 \item[(i)] $Hom(N_i, N_j) = 0$ for all $1\leq i \lneqq j \leq m$;
 \item[(ii)] The sequence is maximal in $\mathcal{G}(N)$, where $N = N_1 \oplus N_2 \oplus \dots \oplus N_m$.
 By maximal we mean that no other bricks can be inserted into $N_1, N_2, \dots , N_m$ preserving (i);
 \item[(iii)] $\mathcal{G}(N) = \mathcal{A}$.
 \end{enumerate}
\end{Definition}

Note that  \cite{Ig2} (page 4) claims that if the sequence $N_1, N_2, \dots , N_m$ satisfies Definition \ref{cfho} (i), then the condition (ii) in this definition is equivalent to the fact that for all $k$, $$\mathcal{G}(N) \cap (N_1 \oplus \dots \oplus N_k)^{\bot} \cap  {}^{\bot}(N_{k+1} \oplus \dots \oplus N_m) = 0.$$

\begin{Corollary}\label{cgx}
Let $M_1, M_2, \dots , M_m$ be a complete forward hom-orthogonal sequence in $\mathcal{A}$, and let
 $M_0 = 0 = M_{m+1}$, $X_i = M_0 \oplus M_1 \oplus \dots \oplus M_i$ and $Y_i = M_{i+1} \oplus \dots
 \oplus M_m \oplus M_{m+1}$ for $0 \leq i \leq m$. Then $\mathcal{G}(X_i) = {}^{\bot}Y_i$ for every $0 \leq i \leq m$.
\end{Corollary}
\begin{proof}
Since $M_1, M_2, \dots , M_m$ is a complete forward hom-orthogonal sequence, we have $Hom(X_i, Y_i) = 0$
 and $\mathcal{C}_i = X^{\bot}_i \cap {}^{\bot}Y_i = \mathcal{A} \cap X^{\bot}_i \cap {}^{\bot}Y_i =
 \mathcal{G}(M) \cap X^{\bot}_i \cap {}^{\bot}Y_i =0$. It follows from Proposition \ref{gx} that
 $\mathcal{G}(X_i) = {}^{\bot}Y_i$.
\end{proof}

 In \cite{Ig2}, Igusa also proved the following property of complete forward hom-orthogonal sequences, which
 shows simple objects are important ingredients in a complete forward hom-orthogonal sequence.

\begin{Lemma}[\cite{Ig2}]\label{sim}
 Let $N_1, N_2, \dots , N_m$ be a complete forward hom-orthogonal sequence in $\mathcal{A}$.
 Then the sequence contains all simple objects (up to isomorphism) in $\mathcal{A}$. Moreover $N_1$ and $N_m$
 are simple objects.
\end{Lemma}

\begin{Corollary}
If $\mathcal{A}$ admits a complete forward hom-orthogonal sequence, then there are only
 finite simple objects in $\mathcal{A}$ up to isomorphism.
\end{Corollary}

\subsection{Maximal green sequences and CFHO sequences}
Minimal extending objects for a torsion class were introduced by Barnard, Carroll, and Zhu in \cite{BCZ}
 to study covers of the torsion class.
\begin{Definition}[\cite{BCZ}]\label{MEO}
Suppose $\mathcal{T}$ is a torsion class in $\mathcal{A}$. An object $M$ in $\mathcal{A}$ is called
 a {\bf minimal extending object} for $\mathcal{T}$ provided with the following conditions:
 \begin{enumerate}
 \item[(i)] Every proper factor of $M$ is in $\mathcal{T}$;
 \item[(ii)] If $0 \rightarrow M \rightarrow X \rightarrow T \rightarrow 0$ is a non-split exact sequence with
     $T\in \mathcal{T}$, then $X\in \mathcal{T}$;
 \item[(iii)] $\mathrm{Hom}(\mathcal{T}, M) = 0$.
 \end{enumerate}
\end{Definition}

Note that if $M$ is a minimal extending object for a torsion class $\mathcal{T}$, then $M$ is indecomposable
 by Definition \ref{MEO} (i). Moreover, assuming (i), then (iii) is equivalent to the fact that
 $M\notin \mathcal{T}$. We write $[M]$ for the isoclass of the object $M$, $ME(\mathcal{T})$ for the set
 of isoclasses $[M]$ such that $M$ is a minimal extending object for $\mathcal{T}$, and $Filt(\mathcal{T}
 \cup \{M\})$ for the iterative extension closure of $Filt(\mathcal{T})\cup M$. The following results
 was proved for the category of finitely generated modules over a finite-dimensional algebra in \cite{BCZ}.
 The results in Section 2 of \cite{BCZ} also hold for abelian length categories.

\begin{Proposition}[\cite{BCZ}]\label{pop} Suppose $\mathcal{T}$ is a torsion class in $\mathcal{A}$ and
 $M$ is an indecomposable object such that every proper factor of $M$ lies in $\mathcal{T}$. Then
 $Filt(\mathcal{T}\cup \{M\})$ is a torsion class and $M$ is a brick.
\end{Proposition}

The following result was proved for finite dimensional algebras in \cite{BCZ}. We give a new proof for abelian length categories.
\begin{Lemma}[\cite{BCZ}]\label{lem}
Let $\mathcal{T}$ be a torsion class in $\mathcal{A}$ and $M\notin \mathcal{T}$ be an indecomposable object in $\mathcal A$ such
that each proper factor of $M$ is in $\mathcal{T}$. Let $N\in Filt(\mathcal{T}\cup \{M\})\backslash \mathcal{T}$ such that each proper
factor of $N$ lies in $\mathcal{T}$. If $Filt(\mathcal{T}\cup \{M\}) \gtrdot \mathcal{T}$, then $M \cong N$.
\end{Lemma}
\begin{proof}
It is clear that $N$ is indecomposable. By Proposition \ref{pop}, the full subcategory $Filt(\mathcal{T}\cup \{N\})$
is a torsion class satisfying that $\mathcal{T} \subsetneq Filt(\mathcal{T}\cup \{N\}) \subset Filt(\mathcal{T}\cup \{M\})$,
which implies that $Filt(\mathcal{T}\cup \{N\}) = Filt(\mathcal{T}\cup \{M\})$ since $Filt(\mathcal{T}\cup \{M\}) \gtrdot \mathcal{T}$.

We claim that $Hom(M, N) \neq 0$ and $Hom(N, M) \neq 0$. Note that $Hom(\mathcal{T}, M) = 0$,
 since each proper factor of $M$ is in $\mathcal{T}$ and $M\notin \mathcal{T}$.  If $Hom(N, M) = 0$,
 then it is easy to see that $Hom(Filt(\mathcal{T}\cup \{N\}),\,\, M)=0$. This contradicts to the fact
 that $M\in Filt(\mathcal{T}\cup \{M\}) = Filt(\mathcal{T}\cup \{N\})$. Then $Hom(M, N) \neq 0$. Similarly,
 we have that $Hom(N, M) \neq 0$.

 Suppose that $M \ncong N$.  Let $f: M\rightarrow N$ and $g: N\rightarrow M$ be two nonzero morphisms.
  Then $f$ and $g$ are not epimorphisms. Otherwise, one would be a proper factor of the other,
 which contradicts to the facts that $M\notin \mathcal{T}$ and $N\notin \mathcal{T}$. Thus $cokerf$ is a proper factor of $N$ and therefore belongs to $\mathcal{T}$.

 If $f$ is not a monomorphism, then $Imf$ is a proper factor of $M$. Then $Imf$ and $cokerf$ belong to $\mathcal{T}$,
  that implies that $M\in \mathcal{T}$, which contradicts to $M\not\in \mathcal{T}$. Hence $f$ is a monomorphism and similarly $g$ is also a monomorphism.

 Note that $gf \neq 0$, since $f\not=0$ and $g$ is a monomorphism. Therefore $gf: M\rightarrow M$ is
 an isomorphism since $M$ is a brick. This implies $g$ is an epimorphism, which is a
 contradiction.

 Thus $M\cong N$.
\end{proof}
\begin{Theorem}[\cite{BCZ}]\label{mini}  Suppose $\mathcal{T}$ is a torsion class in $\mathcal{A}$. Then the map
                           $\eta_{\mathcal{T}}\,:\, [M]\,\mapsto\,Filt(\mathcal{T}\cup \{M\})$
                           is a bijection from the set  $ME(\mathcal{T})$ to the set of $\mathcal{T}'$
                           such that $\mathcal{T} \lessdot \mathcal{T}'$. Moreover, for each
                           such $\mathcal{T}'$, there exists a unique indecomposable
                           object $M$ such that $\mathcal{T}' = Filt(\mathcal{T}\cup \{M\})$, and in this case,
                           $M$ is a minimal extending object for $\mathcal{T}$.
                           Furthermore, the map $Filt(\mathcal{T}\cup \{M\}) \mapsto [M]$ is the
                           inverse to $\eta_{\mathcal{T}}$.
\end{Theorem}
In \cite{BCZ}, the statement that $M$ is a minimal extending object for $\mathcal{T}$ in this case was given in the proof of this theorem.

The following results are the main tools for us to construct
 a stability function for a given class of maximal green sequence.
\begin{Theorem}\label{m2}
Suppose that the sequence $N_1, N_2, \dots , N_m$ is a complete forward hom-orthogonal sequence in
 $\mathcal{A}$. Let $\mathcal{G}_i = \mathcal{G}(N_0\oplus N_1 \oplus \dots N_{i})$ for
 each $0 \leq i \leq m$, where $N_0 = 0$. Then,

 (i)\; $\mathcal{G}_i = Filt(N_0, N_1, \dots, N_i)$;

(ii)\;  $N_i$ is a minimal extending object of $\mathcal{G}_{i-1}$ satisfying that $\mathcal{G}_i = Filt(\mathcal{G}_{i-1}\cup \{N_i\})$;

(iii)\; The sequence
 $0 = \mathcal{G}_0\, \lessdot\, \mathcal{G}_1\, \lessdot\, \mathcal{G}_2\, \lessdot \,\dots\,
 \lessdot\, \mathcal{G}_m = \mathcal{A} $ is a maximal green sequence in $\mathcal{A}$.
\end{Theorem}

\begin{proof}
By Corollary \ref{cgx}, we have $\mathcal{G}_i = \mathcal{G}(N_0\oplus N_1 \oplus \dots N_{i}) = {}^{\bot}(N_{i+1}\oplus\dots N_m \oplus N_{m+1})$ ,
 where $N_{m+1}=0$. We will prove (i) and (ii) using induction method. It is obvious that $\mathcal{G}_0 = Filt(N_0) = 0$ and $N_1 \in ME(\mathcal{G}_0)$
 since $N_1$ is a simple object.

Suppose that $\mathcal{G}_i = Filt(N_0, N_1, \dots, N_i)$ and $N_i\in ME(\mathcal{G}_{i-1})$ for $1\leq i \leq j$.
 We claim that $N_{j+1}\in ME(\mathcal{G}_j)$. First note that $Hom(\mathcal{G}_j, N_k) = 0$ for $k>j$ since $\mathcal{G}_j^{\bot} = (N_0\oplus \dots \oplus N_j)^{\bot}$.
 In particular, $Hom(\mathcal{G}_j, N_{j+1}) = 0$. Second, suppose that $N$ is a proper quotient of $N_{j+1}$, then
 it is clear that $N \in {}^{\bot}(N_{j+2}\oplus \dots \oplus N_m)$. If $f\in Hom(N, N_{j+1})$ is nonzero, since $N$
 is a quotient of $N_{j+1}$, then $f$ must be an isomorphism, which is a contradiction. Thus $N\in  \mathcal{G}_j =
 {}^{\bot}(N_{j+1}\oplus \dots \oplus N_m)$. Let $0 \rightarrow N_{j+1} \stackrel{a}{\longrightarrow} X \stackrel{b}{\longrightarrow} T \rightarrow 0$
 be a nonsplit exact sequence with $T\in \mathcal{G}_j$. Then it is enough to prove that $Hom(X, N_{j+1}) = 0$.
 Let $f\in Hom(X, N_{j+1})$. If $fa\neq 0$, it is an isomorphism and $f$ is a section, which is a contradiction.
 Then $fa=0$, and thus $f$ can be factor through $b$. Since $Hom(T, N_{j+1})=0$, then $f=0$. Then $X\in \mathcal{G}_j$ and
 $N_{j+1}\in ME(\mathcal{G}_j)$. And thus, $\mathcal{G}_{j+1} = Filt(\mathcal{G}_j \cup \{N_{j+1}\}) = Filt(Filt({N_0, N_1, \dots , N_{j}}) \cup \{N_{j+1}\}) = Filt({N_0, N_1, \dots , N_{j+1}})$. Then by induction,  (i) and (ii) hold.

Clearly, $\mathcal{G}_0=0$ and $\mathcal{G}_m=\mathcal{A}$. By Theorem \ref{mini} and (ii), we have  $\mathcal{G}_i\lessdot \mathcal{G}_{i+1}$ for any $i$. Then (iii) holds.
\end{proof}

\begin{Theorem}\label{m1}
Let $0 = \mathcal{T}_0\, \lessdot\, \mathcal{T}_1\, \lessdot\, \mathcal{T}_2\, \lessdot \,\dots\,
 \lessdot\, \mathcal{T}_m = \mathcal{A} $ be a maximal green sequence in $\mathcal{A}$. Then there
 exists a sequence of bricks $N_1, N_2, \dots , N_m$ such that
 \begin{enumerate}
 \item[(i)] $N_i$ is a minimal extending object of $\mathcal{T}_{i-1}$ and $\mathcal{T}_i =
              Filt(\mathcal{T}_{i-1} \cup \{N_i\})$ for each $1 \leq i \leq m$.
 \item[(ii)] $\mathcal{T}_i = Filt(N_0, N_1, \dots, N_i) = \mathcal{G}(N_0 \oplus N_1\oplus \dots \oplus N_i)
             = {}^{\bot}(N_{i+1}\oplus \dots \oplus N_m \oplus N_{m+1})$ for each $0 \leq i \leq m$,
             where $N_0 = 0 = N_{m+1}$.
 \item[(iii)] The sequence $N_1, N_2, \dots , N_m$ is a complete forward hom-orthogonal sequence in $\mathcal{A}$.
 \item[(iv)] $Filt(N_i) = \mathcal{T}_i \cap \mathcal{F}_{i-1}$ for each $1 \leq i \leq m$.
 \item[(v)]  Up to isomorphism, each object $X$ in $\mathcal{A}$ admits a unique filtration
              \[0=X_0\subset X_1\subset X_2 \subset \dots \subset X_m =X\]
              such that $X_i/X_{i-1} \in Filt(N_i)$ for each $1 \leq i \leq m$.
 \end{enumerate}
\end{Theorem}
\begin{proof}
(i) By Theorem \ref{mini}, there exist indecomposable objects $N_1, N_2, \dots, N_m$
such that $N_i$ is a minimal
    extending object of $\mathcal{T}_{i-1}$ and $\mathcal{T}_i = Filt(\mathcal{T}_{i-1} \cup \{N_i\})$
    for each $1 \leq i \leq m$. Then due to Proposition \ref{pop} and by the definition of minimal extending objects,  $N_1, N_2, \dots, N_m$ are bricks.

(ii) It is obvious that $\mathcal{T}_0 = Filt(N_0)$ and $\mathcal{T}_1 = Filt(N_0, N_1)$. Then we can
     prove that $\mathcal{T}_i = Filt(N_0, N_1, \dots, N_i)$ by induction. To prove that $\mathcal{T}_i =
     \mathcal{G}(N_0 \oplus N_1\oplus \dots \oplus N_i)$, it is enough to prove that $\mathcal{F}_i =
     (N_0 \oplus N_1\oplus \dots \oplus N_i)^{\bot}$ for each $0 \leq i \leq m$. Assume that $Y \in
     (N_0 \oplus N_1\oplus \dots \oplus N_i)^{\bot}$, it is clear that $Hom(\mathcal{T}_i, Y) = 0$
     since $\mathcal{T}_i = Filt(N_0, N_1, \dots, N_i)$. Thus $Y \in \mathcal{F}_i$. Conversely, if
     $X\in \mathcal{F}_i$, then we have that $Hom(\mathcal{T}_i, X) = 0$ and hence $X \in
     (N_0 \oplus N_1\oplus \dots \oplus N_i)^{\bot}$. Therefore $\mathcal{T}_i = \mathcal{G}(N_0 \oplus N_1\oplus \dots \oplus N_i)$.

      The statement $\mathcal{T}_i = {}^{\bot}(N_{i+1}\oplus \dots \oplus N_m \oplus N_{m+1})$
      will follow from (iii) by Corollary \ref{cgx}.

(iii) At first, by (i), $N_1, N_2, \dots, N_m$ are bricks.

By the above proof of (ii), if $i\leq j-1$, then $N_i \in \mathcal{T}_{j-1} = Filt(N_0,N_1, \dots, N_{j-1})$ for all $i < j$. And by (i), we have $Hom(\mathcal{T}_{j-1}, N_j) = 0$.  Hence,  $Hom(N_i, N_j) = 0$ for all $i < j$. By (ii), we
      have shown that $\mathcal{G}(N_1\oplus \dots \oplus N_m) = \mathcal{T}_m = \mathcal{A}$. Now it is enough
      to prove that for all $i$, $$(N_0 \oplus N_1\oplus \dots \oplus N_i)^{\bot}  \cap {}^{\bot}(N_{i+1} \oplus \dots  \oplus N_m \oplus N_{m+1})=0.$$
       If $0\not=X\in (N_0 \oplus N_1\oplus \dots \oplus N_i)^{\bot}  \cap {}^{\bot}(N_{i+1} \oplus \dots  \oplus N_m \oplus N_{m+1})$, then $X\in \mathcal{F}_i$ and thus $X\notin \mathcal{T}_i$. Hence there exists $k$
      such that $k > i$ and $X\in T_k\backslash T_{k-1}$. We have that $N_k$ is a factor of $X$, i.e., there
      is an epimorphism $X\rightarrow N_k$. Note that $X\in {}^{\bot}(N_{i+1} \oplus \dots  \oplus N_m \oplus N_{m+1})$
      and $k>i$, therefore $Hom(X, N_k) = 0$, which is a contradiction. Thus
      $(N_0 \oplus N_1\oplus \dots \oplus N_i)^{\bot}  \cap {}^{\bot}(N_{i+1} \oplus \dots  \oplus N_m \oplus N_{m+1})=0$
      and the sequence $N_1, N_2, \dots , N_m$ is a complete forward hom-orthogonal sequence in $\mathcal{A}$. We have also
      proved that  $\mathcal{T}_i = {}^{\bot}(N_{i+1}\oplus \dots \oplus N_m \oplus N_{m+1})$ for all $i$.

(iv) Suppose $X$ is a nonzero object in $\mathcal{T}_i \cap \mathcal{F}_{i-1}$, then $X\in \mathcal{F}_{i-1}
     = (N_0 \oplus N_1\oplus \dots \oplus N_{i-1})^{\bot}$ and thus $N_i$ is a subobject of $X$. Therefore there is
     an exact sequence $0\rightarrow N_i \rightarrow X \rightarrow Y \rightarrow 0$. It is clear that $Y\in \mathcal{T}_i$.
     We claim that $Y\in \mathcal{F}_{i-1} = (N_0 \oplus N_1\oplus \dots \oplus N_{i-1})^{\bot}$. Otherwise assume that
     there is a nonzero morphism $h: N_k \rightarrow Y$ with $k<i$. Then we have the following commutative diagram with
     exact rows.
       \[\xymatrix{
  0  \ar[r]^{} & N_i \ar[d]_{1} \ar[r]^{p} & H \ar[d]_{h} \ar[r]^{t} & N_k \ar[d]_{r} \ar[r]^{} & 0  \\
  0 \ar[r]^{} & N_i \ar[r]^{f} & X \ar[r]^{g} & Y \ar[r]^{} & 0   }\]
     If the first row is nonsplit, then $H\in \mathcal{T}_{i-1}$ since $N_k\in \mathcal{T}_{i-1}$ and $N_k\in ME(\mathcal{T})$.
     Then $Hom(H, X) = 0$ implies $f=0$, which is a contradiction. Thus the first row is split, i.e., there is a
     morphism $s : N_k \rightarrow H$ such that $ts = 1_{N_k}$. Then $r = rts = ghs = 0$, which is a contradiction
     to our assumption. Then we have proved that $Y\in \mathcal{T}_i \cap \mathcal{F}_{i-1}$. By induction on the
     length of $X$, we have that $Y\in Filt(N_i)$, hence $X\in Filt(N_i)$.

     Conversely, since $N_i\in \mathcal{T}_i \cap \mathcal{F}_{i-1}$ by (ii) and $\mathcal{T}_i$ and $\mathcal{F}_{i-1}$ are closed under extensions,  it follows that
     $Filt(N_i) \subset \mathcal{T}_i \cap \mathcal{F}_{i-1}$.  Then, $Filt(N_i) = \mathcal{T}_i \cap \mathcal{F}_{i-1}$.

(v) Suppose that $X$ is a nonzero object in $\mathcal{A}$.
    Let $X_m = X$, and let $0 \rightarrow X_{m-1} \rightarrow X_m \rightarrow X_m/X_{m-1} \rightarrow 0$ be the canonical
    sequence of $X$ with respect to the torsion pair $(\mathcal{T}_{m-1}, \mathcal{F}_{m-1})$. Then we have that
    $X_{m-1}\in \mathcal{T}_{m-1}$ and $X_m/X_{m-1} \in \mathcal{T}_{m} \cap \mathcal{F}_{m-1}$. The existence of
    the filtration for $X$  follows easily. Assume that there are two filtrations
     \[0=X_0\subset X_1\subset X_2 \subset \dots \subset X_m =X,\]
     \[0=X'_0\subset X'_1\subset X'_2 \subset \dots \subset X'_m =X,\]
       satisfying that  $X_i/X_{i-1},\,\,  X'_i/X'_{i-1}\in Filt(N_i)$ for each $1 \leq i \leq m$. Note that $m$ is the length of the maximal green sequence. It is easy to
     see that $X_i,\,\, X'_i \in \mathcal{T}_i$ by (ii).

     Note that $Hom(\mathcal{T}_{i-1}, N_i) = 0$ for each $i$.
     Then $Hom(X_{m-1}, X'_m/X'_{m-1}) = 0$ implies the following commutative diagram
       \[\xymatrix{
  0  \ar[r]^{} & X_{m-1} \ar[d]_{} \ar[r]^{} & X_m \ar[d]_{1} \ar[r]^{} & X_m/X_{m-1} \ar[d]_{} \ar[r]^{} & 0  \\
  0 \ar[r]^{} & X'_{m-1} \ar[r]^{} & X'_m \ar[r]^{} &  X'_m/X'_{m-1} \ar[r]^{}  & 0.   }\]
     It is obvious that $X_{m-1} \cong X'_{m-1}$ and $ X_m/X_{m-1} \cong X'_m/X'_{m-1}$. Similarly, we have that
     $X_{i-1} \cong X'_{i-1}$ and $ X_i/X_{i-1} \cong X'_i/X'_{i-1}$ for all $i$. Then the uniqueness follows.
\end{proof}

From the constructions of  maximal green sequences and complete forward hom-orthogonal sequences resepctively in Theorem \ref{m2} and Theorem \ref{m1}, it is obvious
that there is a bijection between complete forward hom-orthogonal sequences
 and maximal green sequences in an abelian length category.

\begin{Corollary}
 If a maximal green sequence exists in $\mathcal{A}$, then there are only finitely many nonisomorphic simple objects in $\mathcal{A}$.
 \end{Corollary}

\begin{Remark}
The relations between complete forward hom-forward orthogonal sequences and maximal green sequences was firstly
 given by Igusa for finite dimensional hereditary algebras \cite{Ig1} or cluster-tilted algebras of finite representation
 type \cite{Ig2}.
 Then they were extended for the category of finitely generated
 modules over an arbitrary finite-dimensional algebra in the appendix of \cite{Kel1} by Demonet.
 Here we show their truth in abelian length categories via a different approach.
\end{Remark}

\section{Condition for maximal green sequences induced by central charges}

\subsection{Maximal green sequences and stability functions}\label{3.1}

Let $\phi: \mathcal{A}^* \rightarrow \mathcal{P}$ be a stability function on $\mathcal{A}$.
Let $p\in\mathcal{P}$, then  $\mathcal{T}_{\geq p}$ and $\mathcal{A}_{\geq p}$ are respectively given by
\[\mathcal{A}_{\geq p} = \{0\} \cup \{M\in \mathcal{A}\,\, |\,\, M\,\, \text{is $\phi$-semistable and}\,\, \phi(M)\geq p \},\]
 \[ \mathcal{T}_{\geq p}\, =\, \{X\,\in\,Obj(\mathcal{A}): \phi(X')\,\geq\, p\,\, \text{for the maximally destabilizing
                          quotient $X'$ of $X$}\} \cup \{0\}.\]
It is obvious that $\mathcal{A}_{\geq r} \subset \mathcal{A}_{\geq s}$ if and only
 if $r \geq s$ , and it is showed in \cite{BST1} that $\mathcal{T}_{\geq p}$ is a torsion class and
  $\mathcal{T}_{\geq p} = Filt(\mathcal{A}_{\geq p})$.

Recall that $\phi$
 is called {\bf discrete at $p\in \mathcal{P}$} if two $\phi$-stable objects $X_1$ and $X_2$  satisfy $\phi(X_1) = \phi(X_2) = p$,
 then $X_1$ is isomorphic to $X_2$. Moreover, we say $\phi$ to be {\bf discrete} if  $\phi$
 is discrete at every $p\in \mathcal{P}$. On the other hand, Brustle, Smith and Treffinger defined
 an equivalence relation on $\mathcal{P}$ by $p\thicksim q$ when $\mathcal{T}_{\geq p} = \mathcal{T}_{\geq q}$.
 We will write $[p]$ for the equivalence class of $p\in \mathcal{P}$. The following important Proposition
 characterises cover relations for torsion classes induced by $\phi$.

\begin{Proposition}[\cite{BST1}]\label{cov}
Let $\phi: \mathcal{A}^* \rightarrow \mathcal{P}$ be a stability function, and let $p, q\in\mathcal{P}$ such that
 $\mathcal{T}_{\geq p} \subsetneq \mathcal{T}_{\geq q}$. Then $\mathcal{T}_{\geq p} \lessdot \mathcal{T}_{\geq q}$ if
 and only if there is no $r\in \mathcal{P}$ such that $\mathcal{T}_{\geq p} \subsetneq  \mathcal{T}_{\geq r} \subsetneq\mathcal{T}_{\geq q}$,
            and $\phi$ is discrete at $q'$ for each $q'\in [q]$.
\end{Proposition}
The following theorem in \cite{BST1} characterizes the stability function inducing a maximal green sequence.

\begin{Theorem}[\cite{BST1}]
Suppose that $\phi: \mathcal{A}^* \rightarrow \mathcal{P}$ is a stability function such that $\mathcal{P}$
 has no maximal element or the maximal element of $\mathcal{A}$ is not in $\phi(\mathcal{A})$. Then $\phi$
 induces a maximal green sequence if and only if $\phi$ is discrete and $\mathcal{P}/\thicksim$ is finite.
\end{Theorem}

We have the following result.

\begin{Theorem}\label{cor}
Let $\phi: \mathcal{A}^* \rightarrow \mathcal{P}$ be a stability function, and let $p, q\in\mathcal{P}$ such that
 $\mathcal{T}_{\geq p} \lessdot \mathcal{T}_{\geq q}$. Then there exists a unique $\phi$-stable object $N$ satisfying that $\phi(N) \in [q]$.
 Moreover $N$ is a minimal extending object for $\mathcal{T}_{\geq p}$. In particular, we have that $q \leq \phi(N) <p$,
 and if $r_1\in\mathcal{P}$ satisfying that $\phi(N)< r_1 \leq p$, then $\mathcal{T}_{\geq r_1} = \mathcal{T}_{\geq p}$.
\end{Theorem}
\begin{proof}
Since $\mathcal{T}_{\geq p} \lessdot \mathcal{T}_{\geq q}$, we have $p > q$ and $\mathcal{A}_{\geq p} \subsetneq \mathcal{A}_{\geq q}$.
 Let $X\in \mathcal{A}_{\geq q}\backslash \mathcal{A}_{\geq p}$ with phase $\phi(X) = r$.
 Note that $X$ admits a quotient $N$ satisfying that $N$ is stable and $\phi(X) = \phi(N) = r$.
 It is obvious that $q \leq r < p$ and hence $\mathcal{T}_{\geq p} \subset \mathcal{T}_{\geq r} \subset \mathcal{T}_{\geq q}$.
 Since $X$ is semistable, the maximal destabilizing quotient of $X$ is itself. Then $X\in \mathcal{T}_{\geq r}$ and
 $X\notin \mathcal{T}_{\geq p}$. Therefore $\mathcal{T}_{\geq p} \subsetneq \mathcal{T}_{\geq r}$ and hence
 $\mathcal{T}_{\geq r} = \mathcal{T}_{\geq q}$, which implies $\phi(N)=r\in [q]$.

Suppose there are two stable objects $N$ and $N'$  with phase $\phi(N)=r$ and $\phi(N')=r'$ satisfying that $r, r' \in [q]$.
 If $r=r'$, then $N \cong N'$ since $\phi$ is discrete at $r\in [q]$ by Proposition \ref{cov}. Otherwise, we may assume that
 $r < r'$. Then we have that $\mathcal{T}_{\geq q} = \mathcal{T}_{\geq r} = \mathcal{T}_{\geq r'}$ and thus
 $N\in \mathcal{T}_{\geq r} = \mathcal{T}_{\geq r'}$. By the definition of $\mathcal{T}_{\geq r'}$, we have
 that $r=\phi(N) \geq r'$, which is a contradiction. The uniqueness follows.

We shall prove that every proper factor of $N$ is in $\mathcal{T}_{\geq p}$. Indeed, let $N'$ be a nontrivial factor of $N$,
and let $N''$ be the maximal destabilizing quotient of $N'$ and $N'''$ be the stable quotient of $N''$ with phase $\phi(N''') = \phi(N'') = s$.
Note that $N'''\in \mathcal{T}_{\geq q} = \mathcal{T}_{\geq r}$ implies that $s \geq r$. On the other hand, if $s = r$, then $N \cong N'$, which is
a contradiction. Therefore we have  $s > r$.
If $s \geq p$, then $N' \in \mathcal{T}_{\geq p}$.
If $s < p$, we claim that $\mathcal{T}_{\geq s} = \mathcal{T}_{\geq p}$, and then $N' \in \mathcal{T}_{\geq p}$ also follows.
Indeed, since $r< s < p$, we have that $\mathcal{T}_{\geq p} \subset \mathcal{T}_{\geq s} \subset \mathcal{T}_{\geq r}$.
If $\mathcal{T}_{\geq p} \subsetneq \mathcal{T}_{\geq s}$, we have $\mathcal{T}_{\geq s} = \mathcal{T}_{\geq r}$, which implies
that $s\in [p]$. This contradicts to the uniqueness of $N$. As a consequence, every proper factor of $N$ is in $\mathcal{T}_{\geq p}$.
It is easy to see $N$ is a minimal extending object for $\mathcal{T}_{\geq p}$ by Lemma \ref{lem} and Theorem \ref{mini}.

In particular, if $r_1\in\mathcal{P}$ satisfying that $\phi(N)= r < r_1 \leq p$,
 then $\mathcal{T}_{\geq p} \subset \mathcal{T}_{\geq r_1} \subset \mathcal{T}_{\geq q}$.
 If   $\mathcal{T}_{\geq r_1} = \mathcal{T}_{\geq q}$, then $\mathcal{T}_{\geq p} \lessdot \mathcal{T}_{\geq r_1}$.
 Thus there exists a stable object $M$ satisfying that $\phi(M) \in [r_1]$, and in particular $r < r_1 \leq \phi(M) < p$.
 This contradicts to the uniqueness of $N$. Hence $\mathcal{T}_{\geq r_1} = \mathcal{T}_{\geq p}$.
\end{proof}

For the stability function  $\phi$, suppose that there exist $r_0 > r_1 > \dots > r_m$ in $\mathcal{P}$ such that
\[0 = \mathcal{T}_{\geq r_0} \lessdot \mathcal{T}_{\geq r_1} \lessdot \dots \lessdot \mathcal{T}_{\geq r_m} = \mathcal{A}\]
forms a maximal green sequence.
Assume that $N_i$ is the minimal extending object of $\mathcal{T}_{\geq r_{i-1}}$ such that
$\mathcal{T}_{\geq r_i} = Filt(\mathcal{T}_{\geq r_{i-1}}\cup \{N_i\})$. By Theorem \ref{cor}, we know that $N_i$ is stable,
and we may, without loss of generality, assume that $\phi(N_i) = r_i$.
Recall that for $p\in \mathcal{P}$, the full subcategory $\mathcal{A}_p$ is given by
\[\mathcal{A}_p = \{0\} \cup \{M\in \mathcal{A}\,\, |\,\, M\,\, \text{is semistable and}\,\,\phi(M)=p \}.\]
It is shown in \cite{BST1} that $\mathcal{A}_p$ is a wide subcategory for each $p\in \mathcal{P}$. In particular, we have the following
result.
\begin{Proposition}
With the assumptions and notations above, we have that $\mathcal{A}_{r_i} = Filt(N_i)$ for each $i$ with $1 \leq i \leq m$.
\end{Proposition}
\begin{proof}
Note that $N_i$ is stable with phase $r_i$, then $N_i \in \mathcal{A}_p$.
Since $\mathcal{A}_p$ is a wide subcategory, it is closed under extensions.
Then we have that  $Filt(N_i) \subset \mathcal{A}_p$. On the other hand, if $N\in \mathcal{A}_p$ is nonzero,
then $N$ admits a stable factor $N'$ with phase $\phi(N') = \phi(N) =r_i$. Since $\phi$ is discrete at $r_i$,
then $N' \cong N_i$. Therefore we have a short exact sequence $0\rightarrow L \rightarrow N \rightarrow N_i \rightarrow 0$.
Since $N, N_i\in \mathcal{A}_p$ and $\mathcal{A}_p$ is closed under kernels, then $L\in \mathcal{A}_p$. Note that
$l(L)<l(N)$, we may prove that $N\in Filt(N_i)$ by induction on length of $N$. Thus $\mathcal{A}_{r_i} = Filt(N_i)$.
\end{proof}

Note that $\phi(N_1) > \phi(N_2) > \dots > \phi(N_m)$, and for any $X \in Filt(N_i)$, $X$ is semistable
 with phase $r_i$. Then for each nonzero object in $\mathcal{A}$, the filtration induced by $\phi$
 (see Theorem \ref{ruda}) is the same as that induced by the maximal green sequence
 (see Theorem \ref{m1}(v))  by the uniqueness of the filtration.

\begin{Corollary}
With the assumptions and notations above, we have that
\begin{enumerate}
 \item[(i)]  The set of $\phi$-semistable objects in $\mathcal A$  is equal to $\bigcup_{i=1}^{m}\mathcal A_{r_i}\setminus \{0\}$;
  \item[(ii)]  The set $\{N_1, N_2, \dots, N_m\}$
 is a complete set of nonisomorphic  $\phi$-stable objects in $\mathcal{A}$ up to isomorphism.
 \end{enumerate}
\end{Corollary}
\begin{proof}
 It is enough to show that if $M$ is an arbitrary
 semistable object in $\mathcal{A}$, then  $M\in Filt(N_i)$ for some $1\leq i \leq m$, and thus
 $\phi(M)\in \{r_1, r_2, \dots ,r_m\}$ for some $i$. Since the filtration $0 \subsetneq M$ of $M$ induced by the stability function $\phi$ is the same
 as the one induced by the maximal green sequence, then it is obvious that $M\in Filt(N_i)$ for some $1\leq i \leq m$,
  and thus $\phi(M)\in \{r_1, r_2, \dots ,r_m\}$ for some $i$. In particular, since $\phi$ is discrete
  at each $r_i$, then it is easy to see the set $\{N_1, N_2, \dots, N_m\}$
 is a complete set of nonisomorphic  $\phi$-stable objects in $\mathcal{A}$.
\end{proof}

\subsection{Maximal green sequences induced by central charges}\label{3.2}
Let  $\mathcal{A}$ be an abelian length category. If there is a maximal green sequence in $\mathcal{A}$,
then $\mathcal{A}$ admits finitely many simple objects up to isomorphism. Assume that $S_1, S_2, \dots, S_n$
be the complete set of the isomorphism classes of simple objects in $\mathcal{A}$. Then we know that the Grothendieck
group $K_0(\mathcal{A})$ is isomorphic to $\mathbb{Z}^n$. We may write $[X]\in \mathbb{Z}^n$ for the image of $X\in\mathcal{A}$.
 Note that for $\theta\in \mathbb{R}^n$ and $X\in obj(\mathcal{A})$, we also denote by $\langle\theta,X\rangle$ the inner product $\langle\theta, [X]\rangle$ for simplicity.

\begin{Definition}\label{as}
 Let $\mathcal{T}: 0 = \mathcal{T}_0\, \lessdot\, \mathcal{T}_1\, \lessdot\, \mathcal{T}_2\, \lessdot \,\dots\, \lessdot\, \mathcal{T}_m = \mathcal{A} $
 be a maximal green sequence in the abelian length category $\mathcal{A}$, and  $N_1, N_2, \dots, N_m$ be the corresponding complete
 hom-orthogonal sequence. If there exits vectors $\alpha\in \mathbb{R}^n$ and $\beta\in \mathbb{R}_{>0}^n$ such that
 $$\langle\alpha,N_i\rangle \langle\beta,N_{i+1}\rangle <  \langle\alpha,N_{i+1}\rangle \langle\beta,N_i\rangle$$ that is,
 $\frac{\langle\alpha,N_i\rangle}{\langle\beta,N_i\rangle} < \frac{\langle\alpha,N_{i+1}\rangle} {\langle\beta,N_{i+1}\rangle}$  for all $i$, then  the maximal green sequence $\mathcal{T}$ is said to satisfy {\bf crossing inequalities.}
\end{Definition}

\begin{Theorem}\label{main}
 A maximal green sequence  $\mathcal{T}: 0 = \mathcal{T}_0\, \lessdot\, \mathcal{T}_1\, \lessdot\, \mathcal{T}_2\, \lessdot \,\dots\, \lessdot\, \mathcal{T}_m = \mathcal{A} $ in an abelian length category $\mathcal{A}$ can be induced by some central charge $Z : K_0(\mathcal{A}) \rightarrow \mathbb{C}$ if and only if
 $\mathcal{T}$ satisfies crossing inequalities.
\end{Theorem}
\begin{proof}
On one hand, suppose that $\mathcal{T}$ satisfies crossing inequalities.
Let $\alpha\in \mathbb{R}^n$ and $\beta\in \mathbb{R}_{>0}^n$ such that
  $\langle\alpha,N_i\rangle \langle\beta,N_{i+1}\rangle <  \langle\alpha,N_{i+1}\rangle \langle\beta,N_i\rangle$ for all $i$.
 We define a central charge \[Z : \mathcal{A} \rightarrow \mathbb{C}\] which is given by
 $Z(X) = \langle\alpha, X\rangle + \mathrm{i}\cdot \langle\beta, X\rangle$,
 where $\mathrm{i} = \sqrt{-1}$ and $\langle\cdot \,, \cdot\rangle$ is the canonical inner product on $\mathbb{R}^n$.
 Since $\langle\beta, X\rangle >0$ for any nonzero object $X$, the complex number $Z(X)$
 is in the strict upper half-space. Then we define a map $\phi : \mathcal{A}^* \rightarrow [0, 1]$ which is given by
 $\phi(X) = \frac{argZ(X)}{\pi}$ for any nonzero object $X$ in $\mathcal{A}$.

   It is obvious that $0 < argZ(X) < \pi$. Note that \[\mathrm{cot}argZ(X) = \frac{\langle\alpha,X\rangle}{\langle\beta,X\rangle}.\]
 For simplicity, we will write $\mathrm{cot}X$ for $\mathrm{cot} argZ(X)$. It is easy to see that
 for any two nonzero objects $X$ and $Y$, $\phi(X)\leq \phi(Y)$ ( $\phi(X)< \phi(Y)$ ) if and
 only if $\mathrm{cot}X \geq \mathrm{cot}Y$ ( $\mathrm{cot}X > \mathrm{cot}Y$, respectively ),
 which is also equivalent to $\langle\alpha,X\rangle \langle\beta,Y\rangle-\langle\alpha,Y\rangle \langle\beta,X\rangle\geq 0$ ( $> 0$, respectively ).
 It is well-known that $\phi$ is a stability function. Indeed, it is obvious that $\phi(X) = \phi(Y)$ if $X \cong Y$.
 On the other hand,  for any exact sequence $0 \rightarrow L \rightarrow M \rightarrow N \rightarrow0$ with $L, M, N \neq 0$,
 we have that
 \[
\left|\begin{array}{cccc}
   \langle\alpha,M\rangle & \langle\alpha,L\rangle \\
   \langle\beta,M\rangle &   \langle\beta,L\rangle\\
    \end{array}\right| =
 \left|\begin{array}{cccc}
   \langle\alpha,L\rangle+\langle\alpha,N\rangle & \langle\alpha,L\rangle \\
    \langle\beta,L\rangle+\langle\beta,N\rangle &  \langle\beta,L\rangle\\
    \end{array}\right|=
 \left|\begin{array}{cccc}
   \langle\alpha,N\rangle & \langle\alpha,L\rangle \\
    \langle\beta,N\rangle &  \langle\beta,L\rangle\\
    \end{array}\right|=
    \left|\begin{array}{cccc}
   \langle\alpha,N\rangle & \langle\alpha,M\rangle \\
    \langle\beta,N\rangle &  \langle\beta,M\rangle\\
    \end{array}\right|,
\]
which implies that $\phi$ satisfies the seesaw property, and then $\phi$ is a stability function on $\mathcal{A}$.
 Since $\frac{\langle\alpha,N_1\rangle}{\langle\beta,N_1\rangle} < \frac{\langle\alpha,N_2\rangle}{\langle\beta,N_2\rangle} <\dots < \frac{\langle\alpha,N_m\rangle}{\langle\beta,N_m\rangle}$
 and $\mathrm{cot} N_i = \frac{\langle\alpha,N_i\rangle}{\langle\beta,N_i\rangle}$ for each $i$, we have that $$\phi(N_1) > \phi(N_2) > \dots > \phi(N_m).$$

 We claim that any nonzero object $X$ in $Filt(N_i)$ is $\phi$-semistable with phase $\phi(X) = \phi(N_i)$.
 Since $\phi$ satisfies the seesaw property, then $\phi(X) = \phi(N_i)$ and thus
 $\mathrm{cot}X = \mathrm{cot}N_i =\frac{\langle\alpha,N_i\rangle}{\langle\beta,N_i\rangle}$ for any nonzero object $X\in Filt(N_i)$.
 Suppose that $L$ is a nontrivial subobject of $X$. By Theorem \ref{m1} (v), there is
 a unique filtration  $L$:
  \[0=L_0\subset L_1\subset L_2 \subset \dots \subset L_m = L\]
 such that $L_j/L_{j-1}\in Filt(N_j)$ for each $1 \leq j \leq m$. Then we may assume that $[L_j/L_{j-1}] = l_j[N_j]$.
  Thus, $[L]=\sum_{j=1}^ml_j[N_j]$, $\langle\alpha,L\rangle = \sum_{j=1}^ml_j\cdot \langle\alpha,N_j\rangle$ and  $\langle\beta,L\rangle = \sum_{j=1}^ml_j\cdot \langle\beta,N_j\rangle$.
 Since $X\in Filt(N_i)= \mathcal{T}_i \cap \mathcal{F}_{i-1}$,
 we get $$L\in \mathcal{F}_{i-1}=(N_0\oplus N_1\oplus \dots \oplus N_{i-1})^{\bot}.$$
 Therefore $L_0 = L_1 = \dots = L_{i-1} = 0$ and thus $l_1 = l_2 = \dots = l_{i-1}= 0$.
  Notice that $\frac{\langle\alpha,N_j\rangle}{\langle\beta,N_j\rangle} \geq \frac{\langle\alpha,N_i\rangle}{\langle\beta,N_i\rangle}$ for all $j \geq  i$, we have that
 \begin{equation}\label{ine}
 \mathrm{cot} L = \frac{\langle\alpha,L\rangle}{\langle\beta,L\rangle}
                = \frac{ \sum\limits_{j=i}^{m}l_j\langle\alpha,N_j\rangle}{\sum\limits_{j=i}^{m}l_j\langle\beta,N_j\rangle}
                \geq \frac{ \sum\limits_{j=i}^{m}l_j\langle\beta,N_j\rangle\langle\alpha,N_i\rangle/\langle\beta,N_i\rangle}{\sum\limits_{j=i}^{m}l_j\langle\beta,N_j\rangle}
                = \frac{\langle\alpha,N_i\rangle}{\langle\beta,N_i\rangle} = \mathrm{cot} X.
 \end{equation}
 Then $\phi(L)\leq \phi(X)$ which implies $X$ is $\phi$-semistable. In particular, if $X = N_i$,
 then we claim that $Hom(N_i, L) = 0$. Otherwise, the composition $N_i \rightarrow L \hookrightarrow N_i$ is
 nonzero and is an isomorphism, which is a contradiction since $L$ is a nontrivial subobject of $N_i$. Then $L\in \mathcal{F}_i =
 (N_0\oplus \dots \oplus N_i)^{\bot}$, and thus $l_i = 0$ and the inequality (\ref{ine}) is strict.
 Then $N_i$ is $\phi$-stable.

By Theorem \ref{m1}, for any nonzero object $X\in \mathcal{A}$, there is a unique filtration of $X$:
 \begin{equation}\label{fil}
 0=X_{i_0}\subsetneq X_{i_1}\subsetneq X_{i_2} \subsetneq \dots \subsetneq X_{i_l} =X
 \end{equation}
 such that $X_{i_j}/X_{i_{j-1}} \in Filt(N_{i_j})$ for each $1 \leq j \leq l$ and $1 \leq i_1 < i_2 < \dots < i_l \leq m$.
 Since $X_{i_j}/X_{i_{j-1}}$ are $\phi$-semistable with phase $\phi(X_{i_j}/X_{i_{j-1}}) = \phi(N_{i_j})$ and
 $\phi(N_{i_1}) > \phi(N_{i_2}) > \dots > \phi(N_{i_l})$,  the filtration (\ref{fil}) of $X$ induced by the maximal
 green sequence is the same as the unique one induced by the stability function $\phi$ as in Theorem \ref{ruda}. Then
 $X_{i_l}/X_{i_{l-1}}$ is the maximally destabilizing quotient of $X$. Hence the phase of the maximally destabilizing
 quotient of each nonzero object in $\mathcal{A}$ is in $\{\phi(N_1), \phi(N_2), \dots , \phi(N_m)\}$. For simplicity,
 we write $\phi(N_i) = r_i$ for each $1 \leq i \leq m$.

For each $r\in [0,1]$, recall that the torsion class $\mathcal{T}_{\geq r}$ induced by $\phi$
 is given by
  \[ \mathcal{T}_{\geq r}\, =\, \{X\,\in\,Obj(\mathcal{A}): \phi(X')\,\geq\, r\,\, \text{for the maximally destabilizing
                          quotient $X'$ of $X$}\} \cup \{0\}.\]
 In the following proof, we prove that $\mathcal{T}_{\geq r_i} = \mathcal{T}_i$ for each $1 \leq i \leq m$.
 Note that $\phi(X') \geq r_i$ if and only if $\phi(X') \in \{r_1, r_2, \dots, r_i\}$.

Suppose that $Y$ is an arbitrary nonzero object in $\mathcal{A}$, by considering the unique filtration of $Y$,
 it is easy to see that $Y \in \mathcal{T}_k\backslash \mathcal{T}_{k-1}$  if and only if $\phi(Y') = r_k$,
 where $Y'$ is the maximally destabilizing quotient of $Y$.

 Let us prove that $\mathcal{T}_{\geq r_i} = \mathcal{T}_i$ for each $1 \leq i \leq m$.  Indeed, if $X\in \mathcal{T}_i\backslash \{0\}$, then
 there exits $j\leq i$ such that $X \in \mathcal{T}_j\backslash \mathcal{T}_{j-1}$. Thus $\phi(X') = r_j \geq r_i$
 where $X'$ is the maximally destabilizing quotient of $X$, and hence
 $X \in \mathcal{T}_{\geq r_i}$. Conversely, if $W\in \mathcal{T}_{\geq r_i}$, then $\phi(W') = r_j \geq r_i$ for some $j\leq i$ where $W'$ is the maximally destabilizing quotient of $W$, and
 hence $W\in \mathcal{T}_j\backslash \mathcal{T}_{j-1}$. Then $W\in \mathcal{T}_i$.
 Thus $\mathcal{T}_{\geq r_i} = \mathcal{T}_i$ for each $1 \leq i \leq m$.

  In particular, we have that
  $\mathcal{T}_{\geq r_m} = \mathcal{T}_m = \mathcal{A}$. On the other hand, take $r_0\in [0,1]$ such that $r_1<r_0<1$,
  it is obvious that $\mathcal{T}_{ \geq r_0} =\mathcal T_0= 0$.

  Now it follows that the maximal green sequence
  $\mathcal{T}_0\, \lessdot\, \mathcal{T}_1\, \lessdot \,\dots\, \lessdot\, \mathcal{T}_m $
  is induced by the stability function $\phi$.

Converesely, suppose that  the maximal green sequence $\mathcal{T}$ is induced by a central charge
  $$Z : K_0(\mathcal{A}) \rightarrow \mathbb{C}$$
  which is given by $Z(X) = \langle\alpha,X\rangle + \mathrm{i}\cdot \langle\beta,X\rangle$, where $\alpha \in \mathbb{R}^n$ and $\beta\in \mathbb{R}_{>0}^n$.
  Then by definition, $\mathcal{T}$ is induced by the stability function $\phi_Z$. Let $N_1, N_2, \dots, N_m$
  be the corresponding complete forward hom-orthogonal sequence. Then we have that $\phi_Z(N_1) > \phi_Z(N_1) > \dots >\phi_Z(N_m)$
  by Theorem \ref{cor}, and thus $\mathrm{cot}argZ(N_1) < \mathrm{cot}argZ(N_2) < \dots < \mathrm{cot}argZ(N_m)$. This implies that
   $$\langle\alpha,N_i\rangle\cdot \langle\beta,N_{i+1}\rangle <  \langle\alpha,N_{i+1}\rangle\cdot \langle\beta,N_i\rangle$$ for each $1 \leq i \leq m-1$. Then $\mathcal{T}$ satisfies crossing inequalities.
\end{proof}

For a given maximal green sequence, it is hard to determine if it is induced by a central charge. Theorem \ref{main} supports a
possible way (but not complete) to construct a central charge which induce the given maximal green sequence.
 In the following, we give an example to show it is operable.
 We refer to \cite{BDP} for basic concepts on $c$-matrices and maximal green sequences of a quiver.

\begin{Example}

Consider the following quiver $Q$ of type $A_4$ and let $A=KQ$ be the path algebra where $K$ is an algebraically closed field.

\[\begin{tikzpicture}[scale=1.3]
\node at (0,0) (11) {$1$};
\node at (1,0) (21) {$2$};
\node at (2,0) (31) {$3$};
\node at (3,0) (41) {$4$};
\path[-angle 90]
	(11) edge (21)
	(21) edge (31)
	(31) edge (41);
\end{tikzpicture}\]

To give a maximal green sequence for $mod(KQ)$, let us consider the maximal green sequence $(2, 1, 4, 1, 2, 3)$ of $Q$ first, and the corresponding mutations of $c$-matrices are given as follows.

\[\begin{split} &\begin{bmatrix} 1&0&0&0\\0&1&0&0\\0&0&1&0\\0&0&0&1 \end{bmatrix} \xrightarrow{\mu_2}\begin{bmatrix} 1&0&0&0\\0&-1&1&0\\0&0&1&0\\0&0&0&1 \end{bmatrix} \xrightarrow{\mu_1}
\begin{bmatrix} -1&0&1&0\\0&-1&1&0\\0&0&1&0\\0&0&0&1 \end{bmatrix} \xrightarrow{\mu_4}\begin{bmatrix} -1&0&1&0\\0&-1&1&0\\0&0&1&0\\0&0&0&-1 \end{bmatrix}\\& \xrightarrow{\mu_3}
\begin{bmatrix} 0&0&-1&0\\1&-1&-1&0\\1&0&-1&0\\0&0&0&-1 \end{bmatrix} \xrightarrow{\mu_1} \begin{bmatrix} 0&0&-1&0\\-1&0&0&0\\-1&1&0&0\\0&0&0&-1 \end{bmatrix} \xrightarrow{\mu_2}
\begin{bmatrix} 0&0&-1&0\\-1&0&0&0\\0&-1&0&0\\0&0&0&-1 \end{bmatrix}
\end{split}\]

Recall that Igusa showed in \cite{Ig1} that for an acyclic quiver, there is a bijection between
maximal green sequences of the quiver and CFHO sequences of its path algebra  over
an algebraically closed field, and precisely the correspondence claims mutated $c$-vectors of a maximal green sequence of the quiver correspond to dimension vectors
of bricks in a  CFHO sequence of the path algebra.

Then the sequence of mutated $c$-vectors above gives a complete forward hom-orthogonal sequence $N_1, N_2, N_3, N_4, N_5, N_6$ in $mod(KQ)$ satisfying that
\[\underline{dim}N_1 = \begin{bmatrix} 0\\1\\0\\0 \end{bmatrix},\,\,\underline{dim}N_2 = \begin{bmatrix} 1\\0\\0\\0 \end{bmatrix},\,\,\underline{dim}N_3 = \begin{bmatrix} 0\\0\\0\\1 \end{bmatrix},\,\, \underline{dim}N_4 = \begin{bmatrix} 1\\1\\1\\0 \end{bmatrix},\,\,\underline{dim}N_5 = \begin{bmatrix} 0\\1\\1\\0 \end{bmatrix},\,\,\underline{dim}N_6 = \begin{bmatrix} 0\\0\\1\\0 \end{bmatrix}.
\]

By Theorem \ref{m2}, the CFHO sequence $N_1, N_2, N_3, N_4, N_5, N_6$ gives a maximal green sequence $0 = \mathcal{T}_0 \lessdot \mathcal{T}_1 \lessdot \mathcal{T}_2 \lessdot \mathcal{T}_3 \lessdot \mathcal{T}_4 \lessdot \mathcal{T}_5 \lessdot \mathcal{T}_6$ of $mod(KQ)$, which is given by $\mathcal{T}_i = Filt(\{N_1, N_2,\dots, N_i\})$.

For the maximal green sequence $\mathcal{T} : \mathcal{T}_0 \lessdot \mathcal{T}_1 \lessdot \mathcal{T}_2 \lessdot \mathcal{T}_3 \lessdot \mathcal{T}_4 \lessdot \mathcal{T}_5 \lessdot \mathcal{T}_6$, we try to find $\alpha \in \mathbb{R}^4$ and $\beta \in \mathbb{R}^4_{>0}$ such that the maximal green sequence $\mathcal{T}$ satisfies the crossing inequalities.

Let us fix a positive integer vector $\beta=(1,1,1,1)^T$. Assume that $\alpha  = (x_1,x_2,x_3,x_4)^T \in \mathbb{R}^4$. Then $\langle\alpha,N_i\rangle\cdot \langle\beta,N_{i+1}\rangle\,\, <\,\,  \langle\alpha,N_{i+1}\rangle \cdot \langle\beta, N_i\rangle $ for $1\leq i \leq 5$, i.e., \[\frac{\langle\alpha,N_1\rangle}{\langle\beta,N_1\rangle} < \frac{\langle\alpha,N_2\rangle}{\langle\beta,N_2\rangle}
< \frac{\langle\alpha,N_3\rangle}{\langle\beta,N_3\rangle} < \frac{\langle\alpha,N_4\rangle}{\langle\beta,N_4\rangle} < \frac{\langle\alpha,N_5\rangle}{\langle\beta,N_5\rangle} < \frac{\langle\alpha,N_6\rangle}{\langle\beta,N_6 \rangle}\] are given by the following inequalities:

 \[ x_2 < x_1 < x_4 < \frac{x_1+x_2+x_3}{3} < \frac{x_2+x_3}{2} <x_3.\]
 There are infinite many solutions of $\alpha$. By Theorem \ref{main}, each of these solutions for $\alpha$ together with $\beta=(1,1,1,1)^T$ can determine a central charge
which can induce the maximal green sequence $\mathcal{T} : \mathcal{T}_0 \lessdot \mathcal{T}_1 \lessdot \mathcal{T}_2 \lessdot \mathcal{T}_3 \lessdot \mathcal{T}_4 \lessdot \mathcal{T}_5 \lessdot \mathcal{T}_6$. For example, we may take
$\alpha = (2,1,20,3)^T$.
\end{Example}

\section{Applications: $c$-vectors and Rotation lemma  }
In this section, we consider module categories of finite dimensional algebras. Let $A$ be a finite dimensional algebra over a field $K$. Let $\tau$ be the AR-translation of $A$ and $n = rank(K_0(A))$. In this section, we only consider finitely generated right $A$-modules and all modules are assumed to be basic.
\subsection{Bricks and $c$-vectors}\label{4.1}
For a module $X$, we denote the number of nonisomorphic indecomposable summands of $M$ by $|M|$. A module  $M$ is said to be {\bf $\tau$-rigid} if
$Hom(M, \tau M) = 0$. A pair of modules $(M, P)$ is an  {\bf $\tau$-rigid pair} if $M$ is $\tau$-rigid, $P$ is projective and $Hom(P,M)=0$. A $\tau$-rigid pair $(M, P)$ is {\bf (almost) $\tau$-tilting} if  $|M|+|P|=n$ (respectively, $|M|+|P|=n-1$).
For every almost $\tau$-tilting pair $(M, P)$, there exist exactly two $\tau$-tilting pair $(M_1,P_1)$ and $(M_2,P_2)$
such that $M$ is direct summand of both $M_1$ and $M_2$, and $P$ is direct summand of both $P_1$ and $P_2$. In this case, we say $(M_1,P_1)$ and $(M_2,P_2)$ are a mutation of each other. Note that if $(M_1,P_1)$ and $(M_2,P_2)$ are a mutation of each other, there is an almost $\tau$-tilting pair $(M, P)$, and exactly two non-isomorphic indecomposbale $\tau$-rigid modules $X$ and $Y$ such that  $(M_1,P_1)$ is obtained from $(M, P)$ and $X$, and $(M_2,P_2)$ is obtained from $(M, P)$ and $Y$. We also say $(M_2,P_2)$ is obtained from $(M_1, P_1)$ by mutating at $X$, and $(M_1,P_1)$ is obtained from $(M_2, P_2)$ by mutating at $Y$, or say that $(M_1,P_1)$ and $(M_2,P_2)$ are a mutation of each other with the exchange pair $(X,Y)$. Given a $\tau$-rigid pair $(M,P)$, we say a module $X$ is a direct summand of $(M,P)$ if $X$ is a direct summand of $M$
or $P$.

Let $K_0(projA)$ be the Grothendieck group of the additive category $projA$, where $projA$ consists of projective $A$-modules. A basis of $K_0(projA)$ is given by the isomorphism classes $P(1), P(2), \dots, P(n)$ of indecomposable projective modules. For a given $\tau$-rigid module $X$ with the minimal projective presentation $\oplus_{i=1}^nP(i)^{b_i(X)} \rightarrow \oplus_{i=1}^nP(i)^{a_i(X)} \rightarrow X \rightarrow 0$, the {\bf $g$-vector} of $X$ is  defined to be $g(X) = (a_1(X)-b_1(X), a_2(X)-b_2(X), \dots, a_n(X)-b_n(X))^T$. Let $(M = \oplus_{i=1}^tM_i, P = \oplus_{j=t+1}^nP_j)$ be an  $\tau$-tilting pair, then the {\bf $g$-matrix} of $(M, P)$ is defined to be the matrix \[G_{(M,P)}=(g(M_1),\dots, g(M_t), -g(P_{t+1}),\dots, -g(P_n))\in M_{n\times n}(\mathbb{Z}).\]
Notice that $G_{(M,P)}$ is invertible over $\mathbb{Z}$ for any $\tau$-tilting pair $(M = \oplus_{i=1}^tM_i, P = \oplus_{j=t+1}^nP_j)$ \cite{AIR}, Fu \cite{Fu} defined the {\bf $c$-matrix} of $(M, P)$ to be the matrix
\[C_{(M,P)} = (G_{(M,P)}^T)^{-1} = (c_1,\dots,c_t, c_{t+1},\dots,c_n).\]

For $1\leq i \leq t$, the $i$-th column vector $c_i$ of the $c$-matrix $C_{(M,P)}$ of the $\tau$-tilting pair $(M,P)$ is called the {\bf $c$-vector} of $M_i$ with respect to $(M,P)$ and denoted it by $c(M_i)_{(M,P)}$, and for $t+1\leq j \leq n$, the $j$-th column vector $c_j$ of the $c$-matrix $C_{(M,P)}$  is called the {\bf $c$-vector} of $P_j$ with respect to $(M,P)$ and denoted it by $c(P_j)_{(M,P)}$.
Given a $\tau$-tilting pair $(M,P)$,  let $X$ be an indecomposbale direct summand of either $M$ or $P$, then either $X=M_i$ for $i\in \{1,\cdots,t\}$ or $X=P_j$ for $j\in\{t+1,\cdots,n\}$. In this case,  we usually denote briefly by $c(X)$  the $c$-vector $c(X)_{(M,P)}$ when there is no confusion in the context.

Fu \cite{Fu} proved that $c$-matrices have the sign-coherent property, i.e., each column vector of a $c$-matrix is either non-negative or non-positive. Let $(M_1,P_1)$ and $(M_2,P_2)$ be a mutation of each other with the exchange pair $(X,Y)$, we say it is a {\bf green mutation} from  $(M_1,P_1)$ to $(M_2,P_2)$ if  $c(X)=c(X)_{(M_1,P_1)}$ is non-positive, otherwise we say it is a {\bf red mutation} from  $(M_1,P_1)$ to $(M_2,P_2)$.

\begin{Definition}
A  {\bf maximal green sequence of $\tau$-tilting pairs} is a finite sequence $(M_1, P_1)$, $(M_2, P_2)$, $\dots$, $(M_m, P_m)$ of $\tau$-tilting pairs such that  $(M_1, P_1)=(0,A)$,  $(M_m, P_m) = ( A, 0)$, $(M_i, P_i)$  and $(M_{i+1}, P_{i+1})$ are a mutation and the mutation from $(M_i, P_i)$  to $(M_{i+1}, P_{i+1})$ is green for each $1\leq i \leq m-1$.
\end{Definition}

Note that by [Theorem 3.1, \cite{DIJ}], for every mutation  $(M_1,P_1)$ and $(M_2,P_2)$, we have either $FacM_1 \lessdot FacM_2$ or $FacM_2 \lessdot FacM_1$.

\begin{Lemma}\label{x}
 Let $(M_1,P_1)$ and $(M_2,P_2)$ be a mutation of each other  with the exchange pair $(X,Y)$.
Suppose that $FacM_1 \subset FacM_2$ and $N$ is the corresponding minimal extending module. Then $c(X) = -\underline{dim}N$ and  $c(Y) = \underline{dim}N$.
\end{Lemma}
\begin{proof}
Assume that $(M = \oplus_{i=1}^t M_i, P=\oplus_{j=t+1}^{n-1}P_j)$ is the almost $\tau$-tilting pair which can be completed to  $\tau$-tilting pairs  $(M_1,P_1)$ and $(M_2,P_2)$.
Since $FacM_1 \subset FacM_2$, we have that $FacM_1=FacM$ and $FacM_2 = {^{\bot}(\tau M)} \cap P^{\bot}$.
Let $\theta_{(M,P)}= \sum_{i=1}^tg(M_i) - \sum_{j=t+1}^{n-1}g(P_j)\in \mathbb{Z}^m$.
Since a nonzero module $X$ is in $M^{\bot}\cap {^{\bot}(\tau M)} \cap P^{\bot}$ if and only if $X$ is  $\theta_{(M,P)}$-semistable,
then $N$ is $\theta_{(M,P)}$-semistable.
Note that there is a unique $\theta_{(M,P)}$-stable module $B$ up to isomorphhism.
If $N$ is not isomorphic to $B$, then $B$ is a proper factor of $N$ since the set of $\theta_{(M,P)}$-semistable modules form an abelian category with only one simple object $B$.  Hence $B\in Fac M_1= FacM$, which contradicts to $B\in M^{\bot}$.
Thus $N$ is $\theta_{(M,P)}$-stable.
By [Theorem 3.5, \cite{T0}], we have that either $c(X) = \underline{dim}N$ or  $c(X)=-\underline{dim}N$. Note that $\langle g(X), c(X)\rangle=1$ if $X$ is a direct summand of $M_1$, and
  $\langle -g(X), c(X)\rangle=1$ if $X$ is a direct summand of $P_1$.
By [Theorem 1.4(a), \cite{AR}], we have that
\[\langle g(X), \underline{dim}N\rangle = dim_KHom(X,N)-dim_KHom(N,\tau X).\]
Therefore if $X$ is a direct summand of $M_1$, then $Hom(X, N) = 0$ since $Hom(FacM_1,N)=0$. In this case, $\langle g(X), \underline{dim}N\rangle \leq 0$ which implies that
$c(X)= -\underline{dim}N$. If $X$ is a direct summand of $P_1$, it is obvious that $\langle g(X), \underline{dim}N\rangle \geq 0$ which also implies that $c(X)= -\underline{dim}N$.

Then  $c(Y) = \underline{dim}N$ follows from [Lemma 3.4, \cite{T0}].
\end{proof}

\begin{Corollary}\label{c} Let $(M_1,P_1)$ and $(M_2,P_2)$ be a mutation of each other  with the exchange pair $(X,Y)$.
Then the following conditions are equivalent:

(1) $FacM_1 \subset FacM_2$.

(2) $c(X) \leq 0$, i.e., the mutation from $(M_1,P_1)$ to $(M_2,P_2)$ is green.

(3) $c(Y) \geq 0$.

\end{Corollary}
\begin{proof}
The equivalence of (2) and (3) follows from Lemma \ref{x} easily.

(1) $\Rightarrow$ (2): If  $FacM_1 \subset FacM_2$, by Lemma \ref{x}, we have $c(X) = -\underline{dim}N \leq 0$.

(2) $\Rightarrow$ (1): If $c(X) \leq 0$. Suppose that $FacM_2 \subset FacM_1$, then  $c(X) \geq 0$ by Lemma \ref{x} which is a contradiction.
\end{proof}

\begin{Theorem}\label{y}
Let $A$ finite dimensional algebra over a field $K$. There is a bijection between the following sets.

(1) The set of maximal green sequences of torsion classes in $modA$;

(2) The set of maximal green sequences of $\tau$-tilting pairs in $modA$.
\end{Theorem}
\begin{proof}
Let $0=\mathcal{T}_0\,\lessdot\, \mathcal{T}_1\, \lessdot\, \mathcal{T}_2\, \lessdot \,\dots\,
                             \lessdot\, \mathcal{T}_m=modA$
be a maximal green sequences of torsion classes in $modA$. By [Proposition 4.9, \cite{BST0}], there exists a sequence
 $(M_0, P_0)$, $(M_1, P_1)$, $\dots$, $(M_m, P_m)$ of $\tau$-tilting pairs such that  $(M_0, P_0)=(0,A)$,  $(M_m, P_m) = ( A, 0)$, $(M_i, P_i)$  and $(M_{i+1}, P_{i+1})$ are a mutation with the exchange pair $(X_i, Y_i)$ such that $\mathcal{T}_i =FacM_i$. It is enough to show that the mutation from $(M_i, P_i)$  to $(M_{i+1}, P_{i+1})$ is green for each $1\leq i \leq m-1$, which follows from Corollary \ref{c}.

 On the other hand, let the sequence
 $(M_0, P_0)$, $(M_1, P_1)$, $\dots$, $(M_m, P_m)$ be  a maximal green sequences of $\tau$-tilting pairs in $modA$. By Corollary \ref{c}, we have that $FacM_i \lessdot FacM_{i+1}$ and it is clear that $0=FacM_0 \lessdot FacM_1 \lessdot \dots \lessdot FacM_m = modA$ is a maximal green sequences of torsion classes in $modA$.
\end{proof}

Since dimension vectors of bricks appeared in a CFHO can be given via $c$-vectors by Theorem \ref{y} and Lemma \ref{x}, the crossing inequalities can be established by $c$-vectors.   Following these results, we obtain the diagram of the relationship among {\em $c$-vectors} and the concepts concerned in this paper as below, where an arrow $ A \rightarrow B$ means that $B$ can be constructed from $A$.
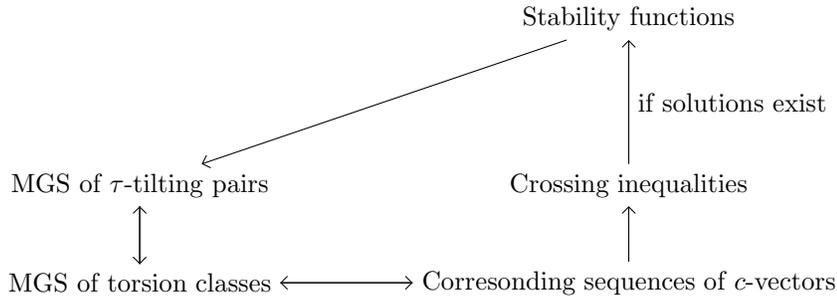
\[\mathord{\begin{tikzpicture}[scale=1.3,baseline=0]
\node at (0,0) (0) {MGS  of torsion classes};
\node at (0,1) (1) {MGS  of $\tau$-tilting pairs};
\node at (5,0) (2) {Corresonding sequences of $c$-vectors};
\node at (5,1) (3) {Crossing inequalities};
\node at (5,2.7) (4) {Stability functions};
\path[-angle 90]
    (0) edge   (2)
	(2) edge   (0)
    (0) edge   (1)
	(1) edge   (0)
         (2) edge   (3)
         (3) edge  node [right] {if solutions exist} (4)
         (4) edge   (1);

\end{tikzpicture}}\]

\subsection{Rotation lemma for Jacobian algebras}\label{4.2}
In this section, all quivers are assumed to be a directed graph with finite arrows and vertices. We refer to the article \cite{DWZ}  for basic concepts and properties of quiver with potentials and Jacobian algebras.
Let $(Q,w)$ be a non-degenerated quiver with potential, and $J(Q, w)$ be the corresponding Jacobian algebra. In this section, we always assume that
Jacobian algebras are finite dimensional. For each vertex $k$, let $(Q',w') = \mu_k(Q,w)$ be the mutation of $(Q,w)$ at $k$. Let $\mathcal{A}= modJ(Q,w)$
and $\mathcal{A'}= modJ(Q',w')$ be the module categories. Assume that the simple modules of $J(Q,w)$ and  $J(Q',w')$ are given by $S_1, S_2, \dots, S_n$ and
$S'_1, S'_2, \dots, S'_n$. We define the following full subcategories of $\mathcal{A}$ (and of $\mathcal{A'}$ accordingly)
\[^{\bot}S_k := \{X\in \mathcal{A}| Hom(X, S_k) =0\};\]
\[S_k ^{\bot}:= \{Y\in \mathcal{A}| Hom( S_k, Y) =0\};\]
and denote $\langle S_k\rangle$ the full subcategories of $\mathcal{A}$ consisting of direct sums of $S_k$.

For every quiver with potential $(Q,w)$ and take any vertex  $k$, we define a matrix $B_k=(b_{ij})$ which is given by $b_{ii} = 1$ for $i\neq k$, $b_{kk}=0$, $b_{kj}$ equals  to the number of arrows from $k$ to $j$ in $Q$ for $j\neq k$, and otherwise $b_{vu}=0$.
 \begin{Theorem}[\cite{BIRS,Mou}]\label{eq}
There are two additive functors (also called the generalized reflection functors) $F_k^+: \mathcal{A} \rightarrow \mathcal{A}'$ and
$F_k^-: \mathcal{A'} \rightarrow \mathcal{A}$ satisfy the following properties.

(1) $F_k^+$ is right exact and  $F_k^-$ is left exact. They are adjoint to each other.

(2) $F_k^+(\langle S_k\rangle)=0$ and $F_k^-(\langle S'_k\rangle )=0$.

(3) $F_k^+(S_k ^{\bot})\subset {^{\bot}S'_k}$ and $F_k^-(^{\bot}S'_k)\subset S_k ^{\bot}$.

(4) The restrictions $F_k^+:S_k ^{\bot}\rightarrow {^{\bot}S'_k}$ and $F_k^-: {^{\bot}S'_k}\rightarrow  S_k ^{\bot}$ are inverse equivalence.
Moreover these equivalences preserve short exact sequences.

(5) If $V\in S_k ^{\bot}$, then $\underline{dim}F_k^+V = B_k\underline{dim}V$, where $\underline{dim}V$
is the dimension vector of $V$ in $\mathcal{A}$ with respect to the basis $[S_1],\dots,[S_n]$ and  $\underline{dim}F_k^+V$
is the dimension vector of $F_k^+V$ in $\mathcal{A}'$ with respect to the basis $[S'_1],\dots,[S'_n]$.
\end{Theorem}

 Now we give the Rotation Lemma for finite dimensional Jacobian algebras, which is a generalization of Igusa's result in \cite{Ig2}, which is the Rotation Lemma for a quiver with potential $(Q,w)$ where $Q$ is mutation equivalent to some quiver of finite type in the sense of \cite{FZ2}.

\begin{Theorem}
With the assumptions and notations above. If $N_1, N_2, \dots, N_m$ is a complete forward Hom-orthogonal sequence in $\mathcal{A} = mod J(Q,w)$ with
$N_1 = S_k$, then $F_k^+(N_2), \dots, F_k^+(N_m), S'_k$ is a complete forward Hom-orthogonal sequence in $\mathcal{A'} = mod J(Q',w')$.
\end{Theorem}
\begin{proof}
By Theorem \ref{eq}, it is obvious that $F_k^+(N_2), \dots, F_k^+(N_m), S'_k$ are bricks.

Since for $2\leq i \leq m$, $N_i\in S_k^{\bot}$, then  $F_k^+(N_i)\in ^{\bot}S'_k$, i.e., $Hom(F_k^+(N_i), S'_k) = 0$. We also have $Hom_{\mathcal{A'}}(F_k^+(N_i), F_k^+(N_j))= Hom_{^{\bot}S'_k}(F_k^+(N_i), F_k^+(N_j)) = Hom_{S_k^{\bot}}(N_i, N_j) = 0$ for $2\leq i <j \leq m$.

In the following, we show that $(F_k^+(N_2)\oplus \dots \oplus F_k^+(N_m)\oplus S'_k )^{\bot} = 0$. To prove this, we claim that if $S$ is a simple module in $\mathcal{A'}$, then $S$ is isomorphic one of modules $F_k^+(N_2), \dots, F_k^+(N_m), S'_k$.

Lemma: If $S \notin \{F_k^+(N_2), \dots, F_k^+(N_m), S'_k\}$ (up to isomorphism), then $Hom(F_k^+(N_i), S) = 0$ for $2 \leq i \leq m$ and $Hom(S'_k, S) = 0$.

Proof: To make the proof more clear, we denote $F_k^+(N_2), \dots, F_k^+(N_m), S'_k$ by $M_1, M_2, \dots, M_m$. If it is not true, we may take the most minimal $i$ from $\{1,2,\dots, m-1\}$ such that $Hom(M_i,S) \neq  0$. In this case, we have that $Hom(S,M_i) =  0$. Otherwise we have $S \simeq M_i$. Now we have that $Hom(M_p,S)= 0$ for $p<i$ and
$Hom(S,M_i) =  0$. Take a nonzero morphism $0 \neq f : M_i \rightarrow S$, then for any morphism $g: S\rightarrow M_j$ for $j>i$, we have that $gf=0$ since $Hom(M_i,M_j)=0$ and hence $g=0$ because $f$ is an epimorphism. Therefore $Hom(S, M_j) = 0$ for $j>i$. Note that $S\in {^{\bot}S'_k}$ (since the simple module $S$ is not isomorphic to $S_k$) and $F_k^+:S_k ^{\bot}\rightarrow ^{\bot}S'_k$ is an equivalence, then
there is a module $R\in S_k ^{\bot}$ such that $S=F_k^+(R)$ and it is clear that $R$ is a brick. Consider the sequence $N_1=S_k$, $\dots$, $N_i$,$R$,$N_{i+1}$, $\dots$,$N_m$, we have that \[R\in S_k ^{\bot},\,\, i.e., \,\,Hom(N_1, R) = 0;\]
\[for\,\, 2\leq p \leq i,\,\, Hom_{\mathcal{A}}(N_p, R) = Hom_{S_k ^{\bot}}(N_p, R) = Hom_{^{\bot}S'_k}(F_k^+(N_p), S) = 0;\]
\[for\,\, i+1\leq p \leq m,\,\,Hom_{\mathcal{A}}(R, N_p) = Hom_{S_k ^{\bot}}(R,N_p) = Hom_{^{\bot}S'_k}(S,F_k^+(N_p)) = 0.\]
This contradicts with the fact that $N_1, N_2, \dots, N_m$ is a complete forward Hom-orthogonal sequence in $\mathcal{A}$. Thus the lemma is true.

Now we can prove our claim that if $S$ is a simple module in $\mathcal{A'}$, then $S \in \{F_k^+(N_2), \dots, F_k^+(N_m), S'_k\}$ (up to isomorphism).
Indeed if $S \notin \{F_k^+(N_2), \dots, F_k^+(N_m), S'_k\}$ (up to isomorphism), by the Lemma above, $Hom(F_k^+(N_i), S) = 0$ for $2 \leq i \leq m$ and $Hom(S'_k, S) = 0$.
Note that $S$ is simple and $S$ is not isomorphic to $S'_k$, we also have that $S\in ^{\bot}S'_k$. Thus there is a module $M\in S_k^{\bot}$ such that $S = F_k^+(M)$.
Then for any $i \geq 2$, we have that
 \[0 = Hom_{\mathcal{A'}}(F_k^+(N_i), S) =  Hom_{\mathcal{A'}}(F_k^+(N_i), F_k^+(M)) = Hom_{^{\bot}S'_k}(F_k^+(N_i), F_k^+(M))=  Hom_{S_k^{\bot}}(N_i,M).\]
Thus $M\in (N_1\oplus N_2\oplus \dots \oplus N_m)^{\bot} = 0$, which is a contradiction. The claim is true. By the claim, it is easy to see that $(F_k^+(N_2)\oplus \dots \oplus F_k^+(N_m)\oplus S'_k )^{\bot} = 0$.

Now we can show that no other bricks can be inserted into the sequence  $F_k^+(N_2), \dots, F_k^+(N_m), S'_k$ preserving the forward Hom-orthogonal
property. Indeed, suppose that there is a brick $M$ in $\mathcal{A}'$ satisfying the condition. Note that $M$ can not be the last one because we have proved that $(F_k^+(N_2)\oplus \dots \oplus F_k^+(N_m)\oplus S'_k )^{\bot} = 0$. Then we may assume that the new sequence satisfying the forward Hom-orthogonal
property is  $F_k^+(N_2), \dots,F_k^+(N_i),M, F_k^+(N_{i+1}),\dots, F_k^+(N_m), S'_k$ and  $M\in ^{\bot}S'_k$. Note that there exists a module $H \in S_k^{\bot}$ such that $M=F_K^+{H}$ and $H$ is a brick. It is easy to see the sequence $N_1, \dots, N_i, H, N_{i+1}, \dots, N_m$ satisfies the forward Hom-orthogonal
property, which is a contradiction. Thus $F_k^+(N_2), \dots, F_k^+(N_m), S'_k$ is a complete forward Hom-orthogonal sequence in $\mathcal{A'}$.

\end{proof}

\begin{Corollary}
Let  $N_1, N_2, \dots, N_m$ is a complete forward Hom-orthogonal sequence in $\mathcal{A} = mod J(Q,w)$ with
$N_1 = S_k$. If there are two vectors $\alpha, \beta\in \mathbb{R}^n$ such that $B_k\beta \in \mathds{R}_{>0}^n$ and
 \[
 \frac{\langle\alpha,N_2\rangle}{\langle\beta,N_2\rangle} <\dots <\frac{\langle\alpha,N_{m}\rangle} {\langle\beta,N_{m}\rangle}<\frac{\langle\alpha,N_{1}\rangle} {\langle\beta,N_{1}\rangle},
 \]
 Then the maximal green sequence $F_k^+(N_2), \dots, F_k^+(N_m), S'_k$ in $\mathcal{A'} = mod J(Q',w')$ can be induced by the central charge
 $Z: K_0(\mathcal{A}') \rightarrow \mathbb{C}$
which is given by \[Z(X) = \langle B_k^T\alpha, [X]\rangle + \mathrm{i}\langle B_k^T\beta, [X]\rangle.\]
\end{Corollary}
\begin{proof}
By Theorem \ref{main}, it is enough to prove that
\begin{equation}\label{z}
 \frac{\langle B_k^T\alpha, F_k^+N_2\rangle}{\langle  B_k^T\beta, F_k^+N_2\rangle} <\dots <\frac{\langle  B_k^T\alpha, F_k^+N_{m}\rangle} {\langle B_k^T\beta,F_k^+N_{m}\rangle}<\frac{\langle  B_k^T\alpha,S'_k\rangle} {\langle  B_k^T\beta, S'_k\rangle}.
\end{equation}

By Theorem \ref{eq}(5), for $2 \leq i \leq m$, since $B_k^2=I$, we have that
\[
 \frac{\langle B_k^T\alpha, F_k^+N_i\rangle}{\langle  B_k^T\beta, F_k^+N_i\rangle} = \frac{\langle B_k^T\alpha, B_kN_i\rangle}{\langle  B_k^T\beta, b_kN_i\rangle}
 =\frac{(\underline{dim}N_i)^T (B_k^T)^2\alpha}{(\underline{dim}N_i)^T (B_k^T)^2\beta} = \frac{(\underline{dim}N_i)^T \alpha}{(\underline{dim}N_i)^T \beta} = \frac{\langle\alpha,N_i\rangle}{\langle\beta,N_i\rangle}.
 \]

On the other hand, by the definition of $B_k$, we have that
\[
\frac{\langle  B_k^T\alpha,S'_k\rangle} {\langle  B_k^T\beta, S'_k\rangle} = \frac{a_k}{b_k} = \frac{\langle\alpha,N_{1}\rangle} {\langle\beta,N_{1}\rangle}.
\]
It is clear that the inequalities (\ref{z}) hold.
\end{proof}

\vspace{5mm}
{\bf Acknowledgements:}\; This project is supported by the National Natural Science Foundation of China (No.11671350) and the Zhejiang Provincial Natural Science Foundation of China (No.LY19A010023).

\end{document}